\documentclass{article}
\usepackage{graphicx} 
\usepackage{amsmath}
\usepackage{amsfonts}
\usepackage{amssymb}
\usepackage{amsthm}
\usepackage{bbm}
\usepackage{tikz}
\usepackage[toc]{appendix}
\usepackage{xcolor}
\usepackage{enumerate}
\usepackage{comment}
\usepackage{float}
\usepackage{authblk}

\usetikzlibrary{shapes.geometric, arrows, positioning}
\tikzstyle{circleNode} = [circle, draw, minimum size=1.5em]
\tikzstyle{arrow} = [->, thick]

\newtheorem{lemma}{Lemma}[section]
\newtheorem{thm}[lemma]{Theorem}
\newtheorem{corollary}[lemma]{Corollary}

\newtheorem{remark}{Remark}
\newtheorem{example}{Example}[section]

\newcommand{\R}{\mathbb{R}}
\newcommand{\N}{\mathbb{N}}
\newcommand{\Z}{\mathbb{Z}}

\newcommand{\RM}{\mathcal{R}}

\title{The Maki-Thompson model with random awareness}

\author{Cristian F. Coletti, Denis A. Luiz and Alejandra Rada}

\begin{document}

\maketitle

\begin{abstract}
    We propose a rumor propagation model in which individuals within a homogeneously mixed population can assume one of infinitely many possible states. To analyze this model, we extend the classical law of large numbers for density-dependent population models in $\mathbb{R}^d$ (Ethier \& Kurtz in 2005, \cite{ethier2009markov}) to an infinite--dimensional setting. Specifically, we prove a convergence result for stochastic processes in the space of $p$-summable real sequences, generalizing the finite-dimensional theory. We apply this framework to an infinite-dimensional continuous-time Markov chain that can be seen as a generalization of the Maki--Thompson model, where each ignorant individual becomes a spreader only after hearing the rumor a random number of times. We derive the asymptotic proportions of individuals in each state at the end of a rumor outbreak. Furthermore, we characterize the maximum proportion of individuals actively spreading the rumor and explore conditions under which the model exhibits waves of rumor propagation.

\end{abstract}

\section{Introduction}\label{S:intro}

In the 1950s and 1960s, the phenomena of rumor spreading bought the attention of mathematicians given its relation with the mechanisms of disease diffusion \cite{goffman1964generalization,rapoport1952mathematical}. In 1964, Daley and Kendal introduced the first mathematical model specific for rumor spreading, \cite{daley1965stochastic}, which gained widespread acceptance in the mathematical community in subsequent years \cite{dunstan1982rumour,lefevre1994distribution,pittel1990daley,sudbury1985proportion}.

The Daley--Kendall model (DK model) considers a population with fixed size, say $n$, divided into three disjoint classes (ignorants, spreaders and stiflers) and the interactions between these individuals change their roles: the interaction ignorant--spreader makes the ignorant become spreader, the interaction spreader--spreader makes both of them become stifler and the spreader--stifler interaction makes the spreader become stifler. This dynamics is driven by a continuous--time Markov chain (CTMC) in a subset of $\Z^3$, namely in $\Omega^{(n)}=\{(z_1,z_2,z_3)\in\Z^3:z_1,z_2,z_3\geq0,z_1+z_2+z_3=n\}$, where the coordinates denote the number of individuals in each class: ignorants, spreaders and stiflers, respectively. The transitions and probabilities of the CTMC in an interval $(t,t+\Delta t)$ are given by
\begin{equation}\label{DKtransition}
\begin{array}{ccc}
\text{transition} \quad &\text{probability} \\[0.2cm]
(z_1,z_2,z_3)\to(z_1-1,z_2+1,z_3) &z_1z_2 \Delta t+o(\Delta t),\\
(z_1,z_2,z_3)\to(z_1,z_2-2,z_3+2) &\binom{z_2}{2}\Delta t+o(\Delta t),\\
(z_1,z_2,z_3)\to(z_1,z_2-1,z_3+1) &z_2z_3\Delta t+o(\Delta t).
\end{array}
\end{equation}

Considering  the increments rather than transitions, the relations in \eqref{DKtransition} are equivalent to
\begin{equation}\label{DKincrements}
\begin{array}{ccc}
\text{increment} \quad &\text{probability} \\[0.2cm]
(-1,1,0) &z_1z_2 \Delta t+o(\Delta t),\\
(0,-2,2) &\binom{z_2}{2}\Delta t+o(\Delta t),\\
(0,-1,1) &z_2z_3\Delta t+o(\Delta t).
\end{array}
\end{equation}

An alternative to the Daley–Kendall model was proposed by Maki and Thompson in their 1973 book \cite{MR0366359}, where they simplified the original stochastic framework by assuming that spreaders become stiflers only upon contact with other spreaders or stiflers. The CTMC of Maki--Thompson model (MT model) has the following increments and probabilities
\begin{equation*}
\begin{array}{ccc}
\text{increment} \quad &\text{probability} \\[0.2cm]
(-1,1,0) &z_1z_2 \Delta t+o(\Delta t),\\
(0,-1,1) &(n-z_1)z_2\Delta t+o(\Delta t).
\end{array}
\end{equation*}

The MT model also became important to the theory and is still considered in the literature \cite{belen2011classical,nekovee2007theory,xiao2019rumor,zhang2022dynamics}.

A generalization of the MT model proposed in \cite{rada2021role}, called \textit{$k$-spreading Maki--Thompson model} ($k$--MT model) is represented in Figure \ref{kMTdiagram}. It is defined subdividing a population of size $n$ into $k+2$ classes of individuals: ignorants, $i$-aware individuals with $i\in\{1,\ldots,k-1\}$, spreaders and stiflers. For $i<k$, an $i$-aware individual has heard the rumor exactly $i$ times without spreading it, while a spreader actively disseminates the rumor after $k$ exposures. A stifler is a former spreader who has ceased to spread the rumor. For $t\geq 0$, we denote the number of individuals in each of these classes at time $t$ by $X^{n}(t)$, $Y^{n}_i(t)$, $Y^{n}(t)$, and $Z^{n}(t)$, respectively.

Thus, the $k$--MT model is the $(k+1)$-dimensional CTMC $\big\{\big(X^{n},Y^{n}_1, \ldots,\allowbreak Y^{n}_{k-1}, Y^{n}\big)(t)\big\}_{t\in[0,\infty)}$ with the following increments and rates:
\begin{equation}\label{EQ:transitions}
{\allowdisplaybreaks
\begin{array}{ccc}
\text{increment} \quad &\text{probability} & \\[0.2cm]
- {\bf e}_1 + {\bf e}_2 \quad &X Y\Delta t+o(\Delta t),& \\[0.2cm]
- {\bf e}_{i} + {\bf e}_{i+1} \quad &Y_i Y\Delta t+o(\Delta t),& 2\leq i\leq k, \\[0.2cm]
- {\bf e}_{k+1} \quad &\left(n-1-X-\displaystyle\sum_{i=1}^{k-1}Y_i \right) Y\Delta t+o(\Delta t),&
\end{array}}%
\end{equation}
where $\{{\bf e}_1, {\bf e}_2,\ldots,{\bf e}_{k+1}\}$ is the natural basis of the $(k+1)$-dimensional Euclidean space.  

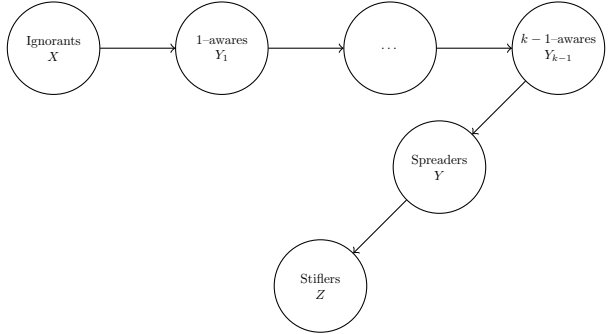
\begin{figure}[H]
    \centering
    \begin{tikzpicture}[scale=0.4, node distance=1cm, every node/.style={circle, minimum size=1.5em, align=center, scale=0.5}]
    \node[draw, minimum height=4em, minimum width=7em] (X) {Ignorants \\$X$};
    \node[draw, minimum height=4em, minimum width=7em, right=of X] (Y1) {$1$--awares \\ $Y_1$};
    \node[draw, minimum height=4em, minimum width=7em, right=of Y1] (Ydots) {$\dots$};
    \node[draw, minimum height=4em, minimum width=7em, right=of Ydots] (Yk1) {$k-1$--awares \\ $Y_{k-1}$};
    \node[draw, minimum height=4em, minimum width=7em, below left=of Yk1] (Y) {Spreaders \\ $Y$};
    \node[draw, minimum height=4em, minimum width=7em, below left=of Y] (Z) {Stiflers \\ $Z$};

    \draw[->] (X) -- (Y1);
    \draw[->] (Y1) -- (Ydots);
    \draw[->] (Ydots) -- (Yk1);
    \draw[->] (Yk1) -- (Y);
    \draw[->] (Y) -- (Z);

    \end{tikzpicture}%
    
    \caption{Scheme of transitions for the $k$-spreading Maki--Thompson model}
    \label{kMTdiagram}
\end{figure}

We propose a generalization of the $k$--MT model by considering $k$ random. That is, a model of which every individual has its own number of times to hear the rumor until spread it. We also add the possibility of an ignorant individual becoming a stifler after hearing the rumor for the first time.

The idea of random dissemination of a rumor in a multi-state population was explored in \cite{lebensztayn2011behaviour}, supposing that every spreader individual has a random number of times to spread the rumor before stopping.  In order to handle with this randomness, the authors embed the infinite-dimensional process in a $2$-dimensional space. In contrast, we handle the infinity of states with a novel technique, as in the next section.

\subsection{Waves of propagation}

In \cite{oliveira2020limit} the authors study the so-called widespread public awareness, which in the rumor context would be the maximum proportion of spreader individuals, that is, the highest peak of the outbreak. Recently, in 2025, the authors in \cite{lebensztayn2025maximum} considered the maximum proportion of spreaders for variants of the Daley--Kendall and Maki--Thompson models. 

Studying the maximum proportion of spreader individuals in our model, we observe that with appropriate choices of parameters, there may be more than one peak of dissemination in the lifetime of the rumor outbreak. That is, for some parameters, there is a moment in which the outbreak may look like it is ceasing, but another peak of propagation is forthcoming without any out-of-model interference (see Examples \ref{p04} and \ref{p06}).

Note that the system of ordinary differential equations for the DK model predicts the existence of one peak until it ends \cite[Ch~11]{ethier2009markov}:
\begin{eqnarray*}
    x'&=&-yx\\
    y'&=&yx-y\bigg(1-x-\frac{y(y-1)}{2}\bigg)
\end{eqnarray*}
may have its time accelerated (as in \cite{lebensztayn2011limit,rada2021role}) to obtain
\begin{eqnarray*}
    \tilde{x}'&=&-\tilde{x}\\
    \tilde{y}'&=&2\tilde{x}+\frac{\tilde{y}(\tilde{y}-1)}{2}-1,
\end{eqnarray*}
which has one local maximum for $\tilde{y}$ for $t\geq0$ and initial condition in $\triangle:=\{(\tilde{x},\tilde{y})\in[0,1]^2:\tilde{x}+\tilde{y}\leq1\}$.

For the MT model we also observe the existence of only one peak, as it can be seen in the accelerated--time system of ordinary differential equations:
\begin{eqnarray*}
    \tilde{x}'&=&-\tilde{x}\\
    \tilde{y}'&=&2\tilde{x}-1.
\end{eqnarray*}

For the rumor model proposed in \cite{lebensztayn2011limit}, the proportion of spreader individuals also has a unique local maximum, since the derivative of the (continuously differentiable) function of proportion of spreader individuals has only one zero with the second derivative negative.

Furthermore, we show that the proportion of spreader individuals in the $k$--MT model, \cite{rada2021role}, has only one maximum for any parameter $k$ chosen (see Corollary \ref{kMTmaximum}).

\subsection{Density-dependent models in $\ell^p(\R)$ space}\label{IDDD}

Although we derive several results on the limit proportions of rumor models, we believe that the main contribution of this paper is not limited to the literature on rumor spreading models. 

In fact, the theory of density-dependent population models is a powerful tool used to derive many rigorous results on rumor spreading models \cite{lebensztayn2011limit,lebensztayn2011behaviour,rada2021role}. We emphasize that we cannot apply this technique to our setup, as the dynamics of our model cannot be embedded into $\Z^d$ for any $d\geq1$. In order to address this problem, we establish a novel functional law of large numbers for an infinite--dimensional version of the former density-dependent population model proposed by Ethier and Kurtz in \cite[Ch.~11]{ethier2009markov}. 

As the size of a population is always finite, if we classify the individuals in a population into countably infinite classes or states, we may see the state of the population as a quasi-null sequence of positive integers, with every coordinate of the sequence being the number of individuals assuming the related class. Immediately, we see that the sequence itself is $p$-summable for any $p>0$, so we may consider the dynamics of the process in the space of $p$-summable sequences of real numbers, $\ell^p(\R)$.

Provided that our process lies in the space $\ell^p(\R)$ and, as we shall see, the set of possible transitions is countable, we may establish a general functional law of large numbers for this setup. As a corollary, we obtain all of our results on rumor spreading models. It is worth mentioning that this result has the former as a particular case.

\subsection{Organization of the paper}

In Section \ref{S:model}, we give the notations and basic definitions, formally define the model, and state the results for the model. In Section \ref{S:examples}, we exhibit two examples with multiple waves of propagation, after that, we apply the results to some known probability distributions and discuss some consequences therein. In Section \ref{S:proofs} we prove our results: in \ref{llnlp} we establish the law of large numbers for processes in $\ell^p(\R)$; and in \ref{proof2.1} and \ref{proof2.2} we prove the results for our model.

\section{Notation, Model and Results}\label{S:model}
\subsection{Notation}
We set $\N_\ast:=\{0,1,\dots\}$ and $\R_+$ as the nonnegative real numbers. For a metric space $\mathbb{X}$, a function $f:\R_+\!\to\mathbb{X}$ is said to be a càdlàg if it is right-continuous with left limits.

\subsection{Definition of the model}

Consider a closed, homogeneously mixed population with $n$ individuals. Any individual will belong to one of the following classes:
 \textbf{spreaders} (those who spread the rumor); \textbf{stiflers} (those who already know the rumor and will not spread it); and the \textbf{$k$-listeners}, $k\in\N$, (who have heard the rumor $k-1$ times without spreading it yet).

For $t\geq 0$, the number of spreaders, $k$-listeners, and stiflers at time $t$ is denoted by $Y^{n}(t)$, $X_k^{n}(t)$, and $Z^{n}(t)$, respectively. On some occasions, the super index $n$ will be suppressed to avoid reloading the notation. Since the population size is $n$, it holds $Z=n-\sum_{k\geq1}X_k-Y$.

Set $p_0$ as the probability that an individual is anti-gossip, that is, someone who if eventually listens to the rumor will never spread it. For $i\geq1$, let $p_i$ denote the probability that an individual who has heard the rumor exactly $i$ times subsequently spreads it. Moreover, we define $q_1=p_1$ and for $i\geq2$, let $q_i=p_i/\sum_{k\geq i}p_k$ be the probability that an individual who has already heard the rumor $i-1$ times becomes spreader after the next contact with someone who tells the rumor once more.

The Maki-Thompson model with random awareness (MT--RA model) is a sequence of infinite-dimensional Markov Chain $\RM^n(t)=\{(Y^n(t),X_1^n(t),X_2^n(t),\allowbreak\dots\}_{t\geq0}$ on $\N_\ast^{\N_\ast}$ with transitions:
    \begin{equation}\label{ratesMTRA}
        \begin{array}{clr}
             \text{Transition}&\text{Rate}&  \\
            -e_{1} & p_0X_1Y,& \\
            -e_{1}+e_{2} & (1-p_0-p_1)X_1Y,& \\
            -e_{i}+e_{0} & q_{i}X_iY,& i\geq1\\
            -e_{i}+e_{i+1}& (1-q_{i})X_iY,& i\geq1\\
            -e_{0} & \left(n-1-\sum_{i=1}^{\infty}X_i\right)Y.&
        \end{array}
    \end{equation}

    A representation of the MT--RA model is presented in Figure \ref{diagram}.
    
\begin{figure}[H]
    \centering
    \begin{tikzpicture}[scale=0.6, node distance=1cm, every node/.style={circle, minimum size=1.5em, align=center, scale=0.6}]
\node[draw, minimum height=2em, minimum width=4em] (X1) {$1$--listeners \\$X_1$};
\node[draw, minimum height=2em, minimum width=4em, right=of X1] (X2) {$2$--listeners \\$X_2$};
\node[draw, minimum height=2em, minimum width=4em, right=of X2] (X3) {$3$--listeners \\$X_3$};
\node[draw, minimum height=2em, minimum width=4em, right=of X3] (X4) {$4$--listeners \\$X_4$};
\node[draw, minimum height=5.5em, minimum width=4em, right=of X4] (Xdots) {$\dots$};
\node[draw, minimum height=2em, minimum width=4em, below=of X2] (Y) {Spreaders \\$Y$};
\node[draw, minimum height=2em, minimum width=4em, below left=of Y] (Z) {Stiflers \\$Z$};

\draw[->] (X1) -- (X2);
\draw[->] (X2) -- (X3);
\draw[->] (X3) -- (X4);
\draw[->] (X4) -- (Xdots);
\draw[->] (X1) -- (Y);
\draw[->] (X2) -- (Y);
\draw[->] (X3) -- (Y);
\draw[->] (X4) -- (Y);
\draw[->] (Xdots) to [out=210,in=10] (Y);
\draw[->] (X1) -- (Z);
\draw[->] (Y) -- (Z);

\end{tikzpicture}
    \caption{Scheme of transitions of the model}
    \label{diagram}
\end{figure}
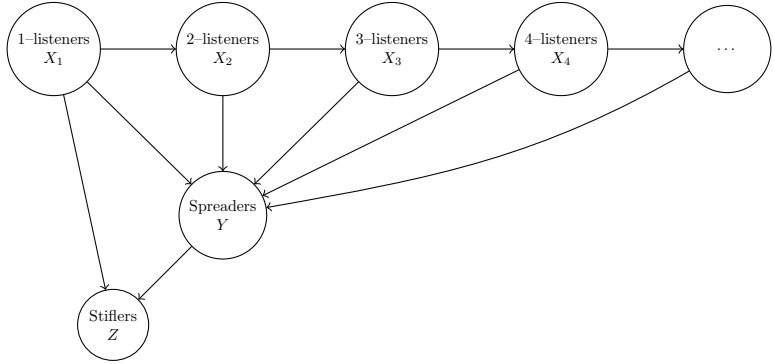

The MT--RA model is said to have initial condition $x^0$ if the following limit holds almost surely.
\begin{equation*}
    \lim_{n\to\infty}\frac{1}{n}(Y^n(0),X_1^n(0),X_2^n(0),\dots)=x^0.
\end{equation*}
The initial condition $x^0=(0,1,0,\dots)$ will be called the standard configuration. 

\subsection{Results}
Let $\tau_n:=\inf\{t:Y^n(t)<0\}$ be the absorption time for the MT--RA model. That is, the time when there is no spreader and the process dies out.

    \begin{thm}\label{LLN}
        Suppose the MT--RA model with initial condition $x^0$ such that $\sum_{l=1}^{\infty}(1+q_l)x_l^0>1$. Then $\lim_{n\to\infty}\tau_n=\tau_\infty\in(0,\infty)$ almost surely. Moreover,
        \begin{equation*}
          \frac{1}{n}(Y^n(\tau_n),X_1^n(\tau_n),X_2^n(\tau_n),\dots)\to(0,x_{1,\infty},x_{2,\infty},\dots),  
        \end{equation*}
        as $n$ goes to infinity, where 
        \begin{eqnarray*}
            x_{1,\infty}&=&x_1^0\exp(-\zeta_\infty),
        \end{eqnarray*}
        and, for $\ell\geq2$,
        \begin{equation}\label{xinfty}
        \begin{array}{rcl}
            x_{\ell,\infty}&=&\exp(-\zeta_\infty)\bigg[x_1^0\frac{(\zeta_\infty)^{\ell-1}}{(\ell-1)!}\Big(1-\sum_{i=0}^{\ell-1}p_i\Big)\\
        &&+\displaystyle\sum_{k=2}^{\ell}\frac{x_k^0}{(\ell-k)!}(\zeta_\infty)^{\ell-k}\prod_{i=k}^{\ell-1}(1-q_i)\bigg],
        \end{array}
        \end{equation}
        where $\zeta_\infty$ is the unique positive solution to
        \begin{equation}\label{eqtau}
            \begin{array}{c}
            \zeta_\infty=\displaystyle(1+q_1)x_1^0\gamma(1,\zeta_\infty)+\sum_{\ell=2}^{\infty}(1+q_\ell)\bigg[x_1^0\frac{\gamma(\ell,\zeta_\infty)}{\Gamma(\ell)}\Big(1-\sum_{i=0}^{\ell-1}p_i\Big)\\
            +\displaystyle\sum_{k=2}^{\ell}x_k^0\frac{\gamma(\ell-k+1,\zeta_\infty)}{\Gamma(\ell-k+1)}\prod_{i=k}^{\ell-1}(1-q_i)\bigg].
            \end{array}
        \end{equation}
        
    \end{thm}
    
    \begin{remark}\label{RLLN}
        Theorem \ref{LLN} holds for any probability distribution with the standard configuration. Also,
        \begin{equation*}
            x_{1,\infty}=\exp(-\zeta_\infty)
        \end{equation*}
        and, for $\ell\geq2$,
        \begin{equation*}
            x_{\ell,\infty}=\exp(-\zeta_\infty)\frac{\zeta_\infty^{\ell-1}}{(\ell-1)!}\Big(1-\sum_{i=0}^{\ell-1}p_i\Big),
        \end{equation*}
        with $\zeta_\infty>0$ being the unique positive root of equation
        \begin{equation*}
            \zeta_\infty=\displaystyle(1+q_1)\gamma(1,\zeta_\infty)+\sum_{\ell=2}^{\infty}(1+q_\ell)\frac{\gamma(\ell,\zeta_\infty)}{\Gamma(\ell)}\Big(1-\sum_{i=0}^{\ell-1}p_i\Big).
        \end{equation*}
    \end{remark}
    Let $\tau_{\max{}}^{(n)}=\sup{\{\zeta\in(0,\tau_n):Y^n(\zeta)\geq Y^n(t)\quad\forall t\in[0,\tau_n]\}}$.
    \begin{thm}\label{Maxprop}
         Under the conditions of Theorem \ref{LLN} it holds that
         \begin{equation*}
            \lim_{n\to\infty}\tau_{\max{}}^{(n)} =\tau_{\max{}}\in[0,\tau_\infty)
        \end{equation*}
        almost surely and
        \begin{equation*}
            \lim_{n\to\infty}\frac{1}{n}Y^n(\tau_{\max{}}^{(n)})\to y_{\max{}}\in[0,1]
        \end{equation*}
        holds almost surely too, with $y_{\max{}}=\max\{y_{j,\max{}}:j=1,\dots,k\}$, where
         \begin{eqnarray}
            y_{j,\max}&=&\displaystyle(1+q_1)x_1^0\gamma(1,\zeta_{j,\max{}})+\sum_{\ell=2}^{\infty}(1+q_\ell)\bigg[x_1^0\frac{\gamma(\ell,\zeta_{j,\max{}})}{\Gamma(\ell)}\Big(1-\sum_{i=0}^{\ell-1}p_i\Big)\nonumber\\
            &&+\displaystyle\sum_{k=2}^{\ell}x_k^0\frac{\gamma(\ell-k+1,\zeta_{j,\max{}})}{\Gamma(\ell-k+1)}\prod_{i=k}^{\ell-1}(1-q_i)\bigg]-\zeta_{\max{}},\label{ymaxj}
        \end{eqnarray}
        where $\zeta_{j,\max{}}$ is a solution in $\zeta$ to the following equation:
         \begin{eqnarray}
             \exp(-\zeta)\bigg((1+q_1)x_1^0+\sum_{\ell=2}^{\infty}(1+q_\ell)\bigg[x_1^0\frac{\zeta^{\ell-1}}{(\ell-1)!}\Big(1-\sum_{i=0}^{\ell-1}p_i\Big)\nonumber&&\\
            +\displaystyle\sum_{k=2}^{\ell}\frac{x_k^0}{(\ell-k)!}\zeta^{\ell-k}\prod_{i=k}^{\ell-1}(1-q_i)\bigg]\bigg)&=&1.\label{zetatime}
         \end{eqnarray}
         Moreover, if either $p_{m+1}\geq\sum_{j=0}^{m}p_j$ for any $m\leq m_0$ and $p_{m}=0$ for any $m\geq m_0+1$ or $p_{m+1}<\sum_{j=0}^{m}p_j$ for any $m\geq2$, the global maximum of the spreader's proportion is also the unique critical point.
    \end{thm}

    The next corollary exhibits a qualitative result for the model in \cite{rada2021role}.

    \begin{corollary}\label{kMTmaximum}
        The limit proportion of spreader individuals in the $k$--MT model has a unique maximum for any $k\geq0$.
    \end{corollary}
    \begin{proof}
        It is enough to see that, in our setting, $p_i=0$ for any $i\neq k$ and $p_k=1$.
    \end{proof}

    \begin{remark}\label{RMaxprop}
        Under the standard configuration, equation \eqref{zetatime} reads
        \begin{eqnarray*}
             \exp(-\zeta)\bigg[(1+q_1)+\sum_{\ell=2}^{\infty}(1+q_\ell)\frac{\zeta^{\ell-1}}{(\ell-1)!}\Big(1-\sum_{i=0}^{\ell-1}p_i\Big)\bigg]=1.
         \end{eqnarray*}
         And equation \eqref{ymaxj} reads
         \begin{eqnarray*}
            y_{j,\max}=\displaystyle(1+q_1)\gamma(1,\zeta_{\max{}})+\sum_{\ell=2}^{\infty}(1+q_\ell)\frac{\gamma(\ell,\zeta_{j,\max{}})}{\Gamma(\ell)}\Big(1-\sum_{i=0}^{\ell-1}p_i\Big)-\zeta_{j,\max{}}.
        \end{eqnarray*}
    \end{remark}
    
\section{Examples}\label{S:examples}

    In order to show some features we pointed out, we exhibit two examples under the standard configuration. In this setup, the distribution $(p_i)_{i\geq0}$ fully characterize the model. 

    \begin{example}\label{p04}
        Consider the MT--RA model with parameters $p_0 = 0.053$, $p_1 = 0.004$, $p_2=0.023$, $p_3=0.163$, $p_4=0.757$ and $p_i=0$ for $i\geq5$. In this case, we observe two waves of propagation in the proportion of spreaders, the second being the most expressive (see Figure \ref{053004023163}).
        \begin{figure}[H]
            \centering
            \includegraphics[width=0.5\linewidth]{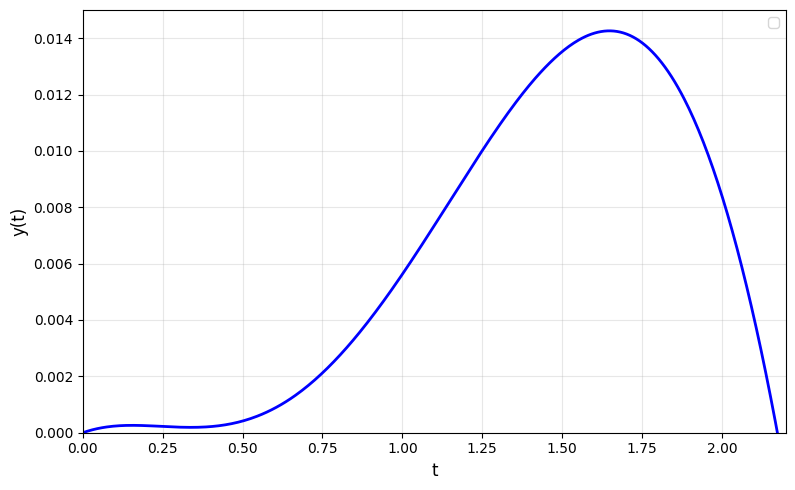}
            \caption{Limit proportion of spreader individuals as a function of accelerated time.}
            \label{053004023163}
        \end{figure}
    \end{example}

    \begin{example}\label{p06}
        Consider the MT--RA with parameters $p_0 = 0.009$, $p_1 = 0.014$, $p_2=0.002$, $p_3=0.038$, $p_4=0.004$, $p_5=0.167$, $p_6=0.766$ and $p_i=0$ for $i\geq7$. Here we observe two waves of propagation in the proportion of spreaders, the most expressive being the first (see Figure \ref{009014002038004167}).
        \begin{figure}[H]
            \centering
            \includegraphics[width=0.5\linewidth]{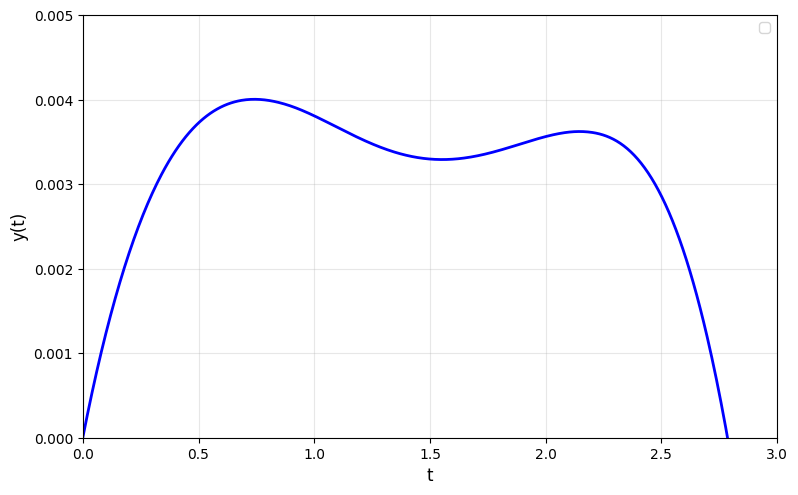}
            \caption{Limit proportion of spreader individuals as a function of accelerated time.}
            \label{009014002038004167}
        \end{figure}
    \end{example}

    In what follows, we apply Theorems \ref{LLN} and \ref{Maxprop} to obtain the final proportions and the maximal proportions for some known distributions in the standard configuration.

    \subsection{Poisson distribution}
    In this case, $p_i$ is the Poisson distribution with mean $\lambda>0$. For this choice of $p_i$, the continuous-time Markov Chain $\{(Y(t),X_1(t),\dots)\}_{t\geq0}$ on $\N_\ast^{\N_\ast}$ has the following transitions and rates:
    \begin{equation*}
        \begin{array}{ll}
             \text{Transition}&\text{Rate}  \\
            -e_{1} & e^{-\lambda}X_1Y, \\
            -e_{1}+e_{0} & e^{-\lambda}\lambda X_1Y,\\
            -e_{1}+e_{2} & (1-e^{-\lambda}-e^{-\lambda}\lambda)X_1Y,\\
            -e_{i}+e_{i+1}& (1-q_i)X_iY,\\
            -e_{i}+e_{0}& q_iX_iY\\
            -e_{0} & Y(n-1-\sum_{i=1}^{\infty}X_i).
        \end{array}
    \end{equation*}
    
    The proportions obtained in Theorems \ref{LLN} and \ref{Maxprop} for particular values of $\lambda$ are given in Table \ref{poitable}.
    
        \begin{table}[H]
            \centering
            \begin{tabular}{|c|c|c|c|c|}
                \hline$\lambda$&$x_{1,\infty}$ & $x_{2,\infty}$ & $x_{3,\infty}$ & $y_{\max}$  \\
                 \hline$2$&$0.238539$ & $0.203074$ & $0.079212$ & $0.093006$\\
                 \hline$16$&$0.005974$&$0.030592$& $0.078317$&$0.000716$\\
                 \hline
            \end{tabular}
            \caption{Final proportions and maximal proportion for the Poisson distribution with parameters $\lambda=2$ and $\lambda=16$.}
            \label{poitable}
        \end{table}
    
    The corresponding graphs for the proportion of spreading individuals of these values of $\lambda$ are presented in Figure \ref{graphpoisson}.  
    
    \begin{figure}[H]
        \centering
        \includegraphics[width=0.44\linewidth]{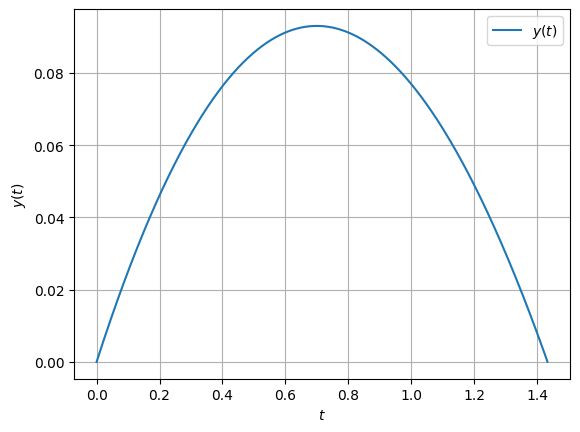}
        \includegraphics[width=0.445\linewidth]{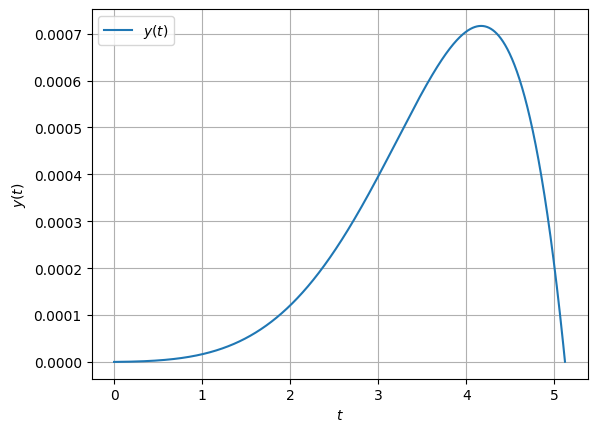}
        \caption{Limit proportion of spreaders as a function of accelerated time with $\lambda=2$ and  $\lambda=16$.}
        \label{graphpoisson}
        \end{figure}

    The maximal proportion of spreader individuals as a function of $\lambda$, $y_{\max}(\lambda)$, seems to reach its maximum at $\lambda\approx1.67$, as shown in Figure \ref{poimaximal}.
    \begin{figure}[H]
        \centering
        \includegraphics[width=0.99\linewidth]{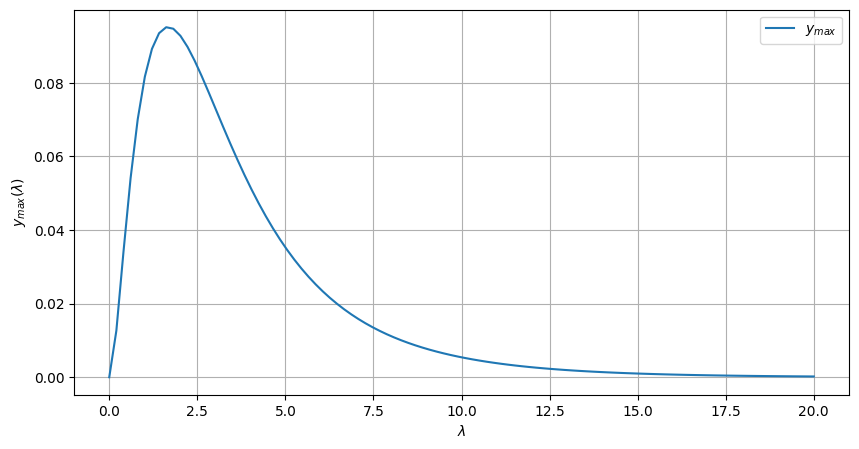}
        \caption{Graph of $y_{max}(\lambda)$ for the Poisson distribution.}
        \label{poimaximal}
    \end{figure}

    We conjecture that $y_{\max}(\lambda)$ reaches its maximum for some $\lambda>0$; however, proving this statement and its uniqueness is beyond the scope of this work and is left as an open question.

    Now, consider a rumor outbreak under the standard configuration that achieves the maximal proportion of spreader individuals greater than $0.1$. The graph in Figure \ref{poimaximal} makes us suspicious that it is not possible to use the MT--RA model with Poisson distribution to analyze it.

    The final proportion of ignorant individuals, $x_{1,\infty}$, also depends on $\lambda$ and appears to converge to $0$ as $\lambda$ goes to infinity.

    \begin{figure}[H]
        \centering
        \includegraphics[width=1\linewidth]{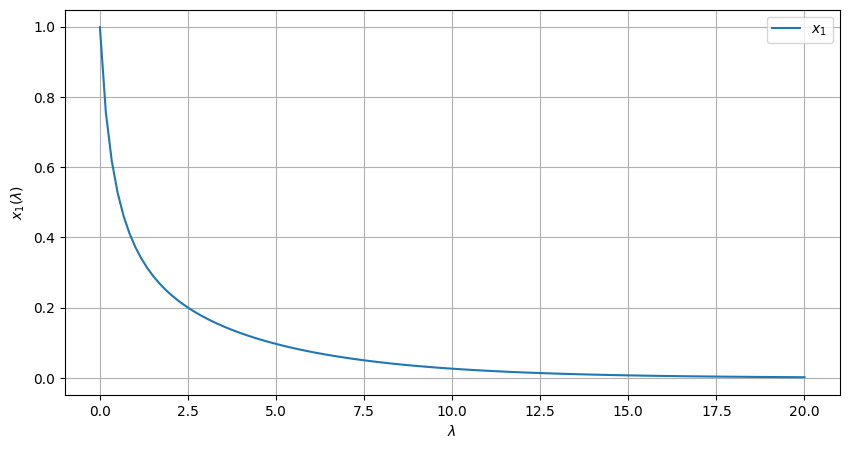}
        \caption{Values of the final proportion of ignorant individuals, $x_{1,\infty}$, as a function of $\lambda$.}
        \label{poix1}
    \end{figure}

    The graph in Figure \ref{poix1} says that the higher the parameter $\lambda$, the lower the number of individuals who never hear the rumor. Combining this with the graph in Figure \ref{poimaximal}, there is a strong suggestion that the increase in the parameter $\lambda$ also increases the time of the rumor outbreak.

\subsection{Zeta distribution}
        Here, $p_i$ is the Zeta distribution with parameter $s>1$, that is, 
        \begin{equation*}
            p_i=\frac{i^{-s}}{\zeta(s)},
        \end{equation*} 
        where $\zeta(s)=\sum_{i\geq1}i^{-s}$.
        
        Since $p_0=0$, there is no transition from ignorant to stifler, then the continuous-time Markov Chain $\{(Y(t),X_1(t),\dots)\}_{t\geq0}$ on $\N_\ast^{\N_\ast}$ for this case has the following transitions and rates:
    \begin{equation}
        \begin{array}{ll}
             \text{Transition}&\text{Rate}  \\
            -e_{1}+e_{0} & \frac{1}{\zeta(s)} X_1Y,\\
            -e_{1}+e_{2} & (1-\frac{1}{\zeta(s)})X_1Y,\\
            -e_{i}+e_{i+1}& (1-q_i)X_iY,\\
            -e_{i}+e_{0}& q_iX_iY,\\
            -e_{0} & Y(n-1-\sum_{i=1}^{\infty}X_i).
        \end{array}
    \end{equation}
    
    The proportions obtained in Theorems \ref{LLN} and \ref{Maxprop} for selected values of $s$ are given in Table \ref{zetable}.
    
    \begin{table}[H]
            \centering
            \begin{tabular}{|c|c|c|c|c|}
                \hline$s$&$x_{1,\infty}$ & $x_{2,\infty}$ & $x_{3,\infty}$ & $y_{\max}$  \\
                 \hline$1.01$&$0.169622$ & $0.297948$ & $0.2629893$ & $0.00379$\\
                 \hline$5$&$0.2014504$&$0.0114945$& $0.0014158$&$0.2994$\\
                 \hline
            \end{tabular}
            \caption{Final proportions and maximal proportion for the Zeta distribution with parameters $s=1.01$ and $s=5$.}
            \label{zetable}
    \end{table}

    The graphs corresponding to the proportion of spreader individuals for these choices of $s$ are given in Figure \ref{graphzeta}.

    \begin{figure}[H]
        \centering
        \includegraphics[width=0.46\linewidth]{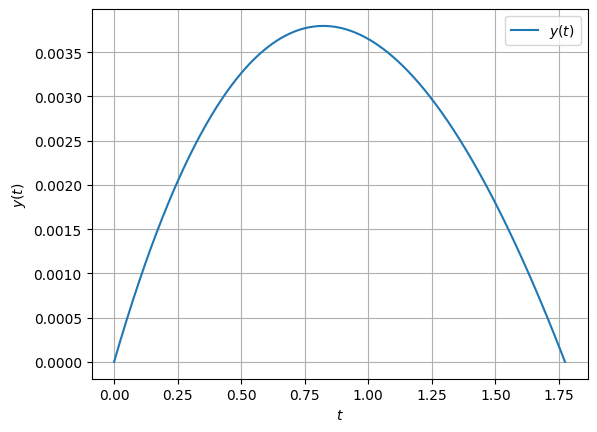}
        \includegraphics[width=0.45\linewidth]{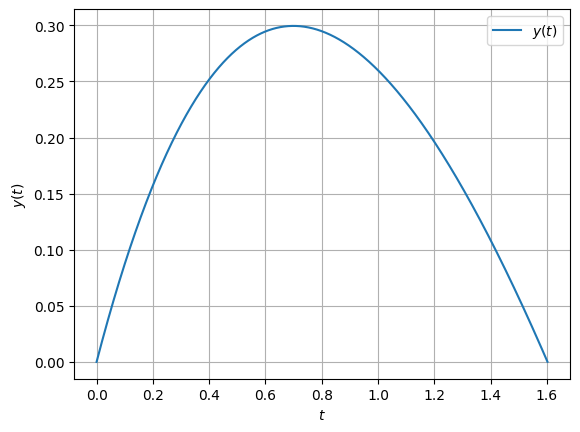}
        \caption{Limit proportion of spreaders as a function of accelerated time with $s=1.01$ and $s=5$.}
        \label{graphzeta}
    \end{figure}
    
    The maximal proportion of spreader individuals, $y_{\max}$, depends on $s$ and appears to converge to $0.3068$ as $s$ goes to infinity (see Figure \ref{zetamax}).
    
    \begin{figure}[H]
        \centering
        \includegraphics[width=1\linewidth]{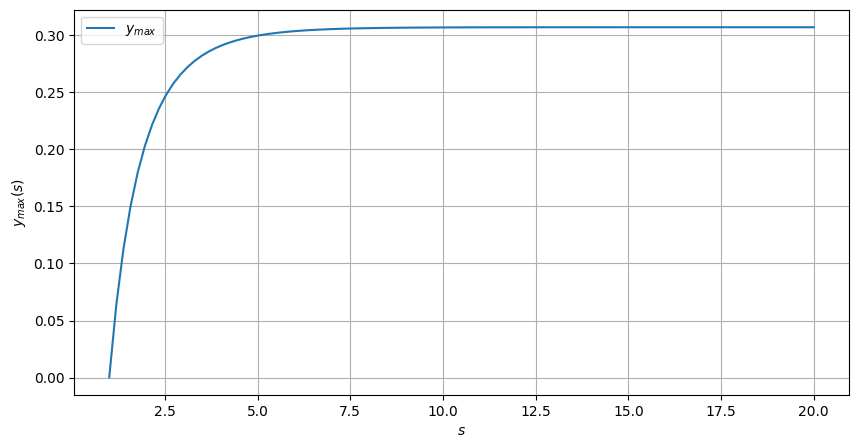}
        \caption{Values of $y_{max}$ as a function of $s$.}
        \label{zetamax}
    \end{figure}

    For $s>2$, the Zeta distribution of parameter $s$ has first moment $\zeta(s-1)/\zeta(s)$. It suggests that for large values of $s$, the rumor outbreak behaves as in the MT model. As a matter of fact, the MT model is a particular case of the MT--RA model with $p_i=0$ for $i\neq1$ and $p_1=1$ so we compute its maximal proportion of spreader individuals and obtain the same limit proportion: $0.3068$.

    The final proportion of ignorant individuals, $x_{1,\infty}$, also depends on $s$ and appears to converge approximately to $0.203187$ as $s$ goes to infinity (see Figure \ref{zetaig}).
    \begin{figure}[H]
        \centering
        \includegraphics[width=1\linewidth]{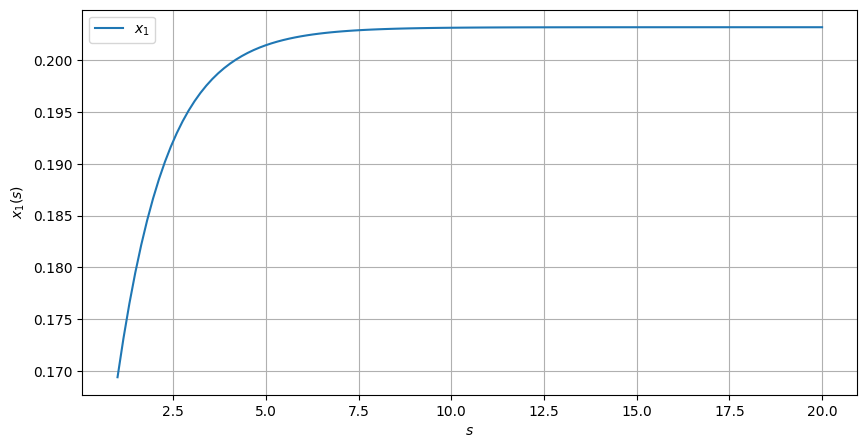}
        \caption{Values of the final proportion of ignorant individuals, $x_{1,\infty}$, as a function of $s$.}
        \label{zetaig}
    \end{figure}

    Again, in the MT model, the final proportion of ignorant individuals is $0.203187$, reinforcing the similarity between both models for large values of $s$.
    
    \subsection{Uniform distribution}
    Here, $p_i$ is the Uniform distribution on $\{0,1,\dots,k\}$, that is,
    \begin{equation*}
        p_i=\frac{1}{k+1},\quad 0\leq i\leq k.
    \end{equation*}
    
    In this case, the MT--RA resembles the $k$--MT model, introduced in \cite{rada2021role}, but these models behave asymptotically differently. In order to show that, we summarize some of the results for the $k$-spreading Maki--Thompson model in Table \ref{kspreading}. We also compute and include the maximal proportion of spreader individuals as it is a particular case of the MT--RA model with $p_i=0$ for $i\neq k$ and $p_k=1$.

    \begin{table}[h]
            \centering
            \begin{tabular}{|c|c|c|c|c|}
                \hline$k$&$x_{1,\infty}$ & $x_{2,\infty}$ & $x_{3,\infty}$ & $y_{\max}$  \\
                 \hline$2$&$0.116586$ & $0.250558$ & $0$ & $0.174233$\\
                 \hline$3$&$0.0680169$&$0.182829$& $0.245723$&$0.110627$\\
                 \hline
            \end{tabular}
            \caption{Final proportions and maximal proportion for the $k$-spreading Maki--Thompson model in \cite{rada2021role} with parameters $k=2$ and $k=3$.}
            \label{kspreading}
    \end{table}

     The MT--RA model with Uniform distribution on $\{0,\dots,k\}$ is the continuous-time Markov Chain $\{(Y(t),X_1(t),\dots,X_{k-1}(t))\}_{t\geq0}$ on $\N_\ast^{k}$ with transitions and rates:
    \begin{equation}
        \begin{array}{ll}
             \text{Transition}&\text{Rate}  \\
            -e_{1} & \frac{1}{k+1}X_1Y, \\
            -e_{1}+e_{0} & \frac{1}{k+1}X_1Y,\\
            -e_{1}+e_{2} & \frac{k-1}{k+1}X_1Y,\\
            -e_{i}+e_{i+1}& \frac{k-i}{k+1-i}X_iY,\\
            -e_{i}+e_{0}& \frac{1}{k+1-i}X_iY\\
            -e_{0} & Y(n-1-\sum_{i=1}^{k-1}X_i).
        \end{array}
    \end{equation}
    
    Now, in Table \ref{unitable}, we exhibit the limit proportions for $k=2$ and $k=3$.
    \begin{table}[H]
            \centering
            \begin{tabular}{|c|c|c|c|c|}
                \hline$k$&$x_{1,\infty}$ & $x_{2,\infty}$ & $x_{3,\infty}$ & $y_{\max}$  \\
                 \hline$2$&$0.3438126$ & $0.1223581$ & $0$ & $0.0848$\\
                 \hline$3$&$0.3186747$&$0.1822157$& $0.0520947$&$0.067350$\\
                 \hline
            \end{tabular}
            \caption{Final proportions and maximal proportion for the Uniform distribution with parameters $k=2$ and $k=3$.}
            \label{unitable}
    \end{table}

    In Figure \ref{graphuniform}, we have the graphs corresponding to the proportion of spreader individuals for these parameters.
    
    \begin{figure}[H]
        \centering
        \includegraphics[width=0.45\linewidth]{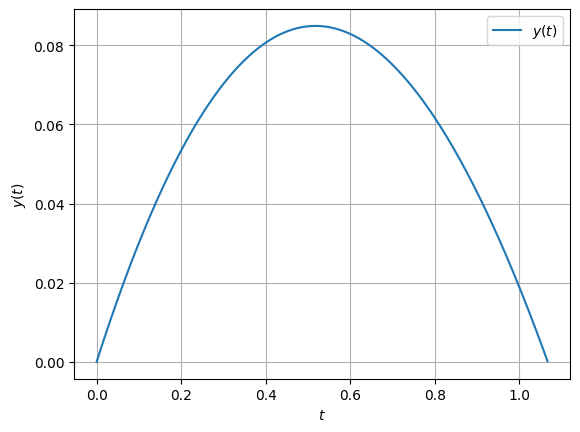}
        \includegraphics[width=0.45\linewidth]{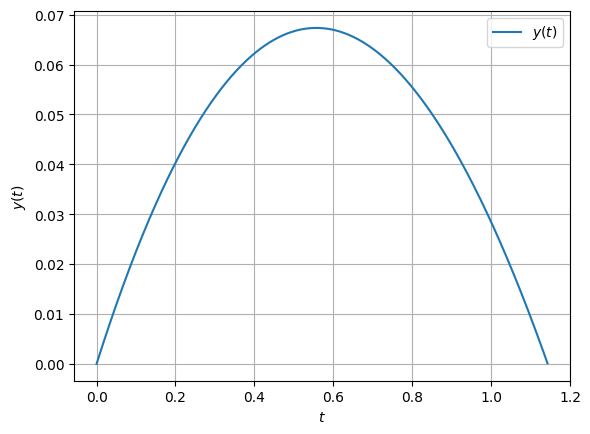}
        \caption{Limit proportion of spreaders as a function of accelerated time with $k=2$ and $k=3$.}
        \label{graphuniform}
    \end{figure}
    
    The maximal proportion of spreader individuals depends on $k$. The graph in Figure \ref{unimax} shows this relation.
    \begin{figure}[H]
        \centering
        \includegraphics[width=1\linewidth]{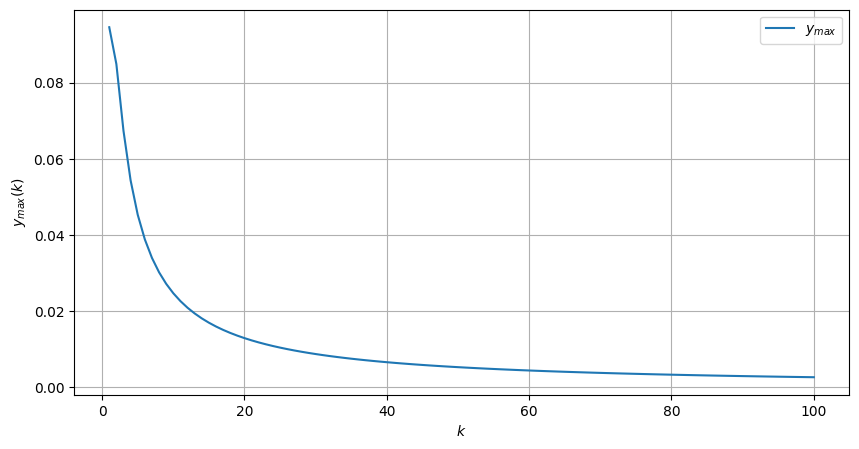}
        \caption{Values of $y_{max}$ as a function of $k$.}
        \label{unimax}
    \end{figure}

    We conjecture that $y_{\max}(k)$ vanishes as $k$ goes to infinity. Moreover, the final proportion of ignorant individuals, $x_{1,\infty}$, also depends on $k$ and appears to converge approximately to $0.3085$ as $k$ goes to infinity (see Figure \ref{unig}).
    \begin{figure}[H]
        \centering
        \includegraphics[width=1\linewidth]{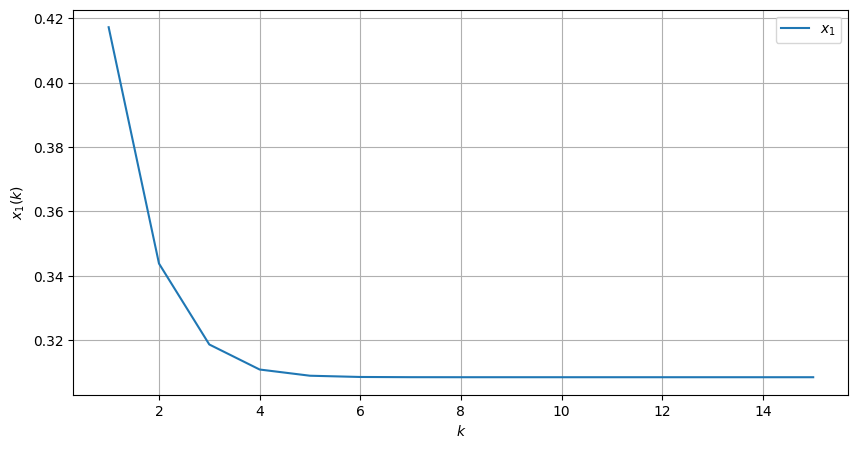}
        \caption{Values of $x_{1,\infty}$ as a function of $k$.}
        \label{unig}
    \end{figure}
    
\section{Proofs}\label{S:proofs}
\subsection{Functional Law of Large Numbers for stochastic processes in $\ell^p(\R)$}\label{llnlp}

    In what follows, some definitions and results from Functional Analysis are needed (see \cite{rudin} for more details). The following construction is inspired by Ethier and Kurtz's ideas in \cite[Ch. 11]{ethier2009markov}.

    For $p\geq1$, let the space of $p$-summable sequences,
    \begin{equation*}
        \ell^p(\R):=\Big\{x=(x_i)_{i\geq1}:x_i\in\mathbb{R}, \Vert x\Vert_p^p:=\sum_{i\geq1}|x_i|^p<\infty\Big\},
    \end{equation*}
    and the set of integer sequences converging to zero,
    \begin{equation*}
        c_{0}(\mathbb{Z}):=\{x=(x_i)_{i\geq1}:x_i\in\mathbb{Z}\quad\forall i\geq1,\lim_{i\to\infty}x_i=0\}.
    \end{equation*}
    In particular, $c_{0}(\Z)\subset\ell^p(\R)$. Let $E\subset\ell^p(\R)$ be an open and bounded set and consider, for any $n\geq1$,
    \begin{equation*}
        E_n=E\cap\{n^{-1}x:x\in c_{0}(\mathbb{Z})\}.
    \end{equation*}
    
    Let $\mathcal{I}$ be a bounded countable subset of $\ell^p(\R)$. For any $m\in\mathcal{I}$, let $\beta_m$ be a nonnegative continuous function such that if $x\in E_n$ and $\beta_m(x)>0$, then $x+n^{-1}m\in E_n$. 

    Assume that there exists $M>0$ such that $\tilde{\beta}_m:=\sup_{x\in E}\beta_m(x)$ satisfies
    \begin{equation}\label{Fcondition}
        \sum_{m\in\mathcal{I}}\Vert m\Vert_p\tilde{\beta}_m\leq M,
    \end{equation}
    and let $F:\ell^p(\R)\to\ell^p(\R)$ be defined pointwise by
    \begin{equation}\label{Fdef}
        F(x)=\sum_{m\in\mathcal{I}}m\beta_m(x).
    \end{equation}
    Condition \eqref{Fcondition} ensures that $F$ is well defined in $E$. Assume that $F$ is Lipschitz. 
    
    Notice that $\beta_m(\gamma(s))$ is integrable over $s\in[0,t]\subset A$ whenever $\gamma:A\subset\R\to E$ is a càdlàg function. Thus, it is well defined the quantity
    \begin{equation*}
        \int_0^tF(\gamma(s))ds:=\int_0^t\sum_{m\in\mathcal{I}}m\beta_m(\gamma(s))ds=\sum_{m\in\mathcal{I}}m\int_0^t\beta_m(\gamma(s))ds\in\ell^p(\R)
    \end{equation*}
    for $[0,t]\subset A$. Indeed,
    \begin{equation*}
        \int_0^tF(\gamma(s))ds\leq Mt.
    \end{equation*}
    
    Let $(X^n:n\geq1)$ be a sequence of stochastic processes in $E_n$ such that:
    \begin{equation}\label{Xndef}
        X^n(t)=X^n(0)+\sum_{m\in\mathcal{I}}n^{-1}mY_m\bigg(n\int_0^t\beta_m\big(X^n(s)\big)ds\bigg),
    \end{equation}
    for $t\geq0$, where $\{Y_m:m\in\mathcal{I}\}$ is a family of independent standard Poisson processes.
    \begin{thm}\label{IDLLN}
         Suppose that $X^n(0)\to X(0)\in E$ almost surely as $n\to\infty$ and define
    \begin{equation}\label{Xdef}
        X(t)=X(0)+\int_0^tF(X(s))ds.
    \end{equation}
    Then for every $t\geq0$, $\lim_{n\to\infty}\sup_{s\leq t}\Vert X^n(s)-X(s)\Vert_p\to0$ as $n\to\infty$ almost surely.
    \end{thm}
    \begin{proof}
        We follow closely the approach used in \cite{ethier2009markov}. First, we rewrite equation \eqref{Xndef} as
        \begin{equation*}
            X^n(t)=X^n(0)+\sum_{m\in\mathcal{I}}m n^{-1}\tilde{Y}_m\bigg(n\int_0^t\beta_m\big(X^n(s)\big)ds\bigg)+\int_0^tF(X^n(s))ds,
        \end{equation*}
        where for any $m$, $\tilde{Y}_m(u)=Y_m(u)-u$ is the Poisson process centered at its expectation. Now, since
        \begin{equation}\label{limsup}
            \lim_{n\to\infty}\sup_{u\leq v}|n^{-1}\tilde{Y}_m(nu)|=0\quad a.s.,\quad v\geq0,
        \end{equation}
        it suffices to show that the Poissonian contribution of the process $X^n$ converges to zero almost surely:
        \begin{equation}\label{ineqepsilon}
            \begin{array}{ccl}
            \varepsilon_n(t)&=&\sup_{u\leq v}\bigg\Vert X^n(u)-X^n(0)-\int_0^uF(X^n(s))ds\bigg\Vert_p\\
            &\leq&\displaystyle \sum_{m\in\mathcal{I}}\Vert m\Vert_p n^{-1}\sup_{u\leq t}|\tilde{Y}_m(n\tilde{\beta}u)|\\
            &\leq&\displaystyle \sum_{m\in\mathcal{I}}\Vert m\Vert_pn^{-1}(Y_m(n\tilde{\beta}_m t)+n\tilde{\beta}_mt).
            \end{array}
        \end{equation}
        By the law of large numbers for Poisson processes,
        \begin{eqnarray*}
            \lim_{n\to\infty}\sum_{m\in\mathcal{I}}\Vert m\Vert_pn^{-1}(Y_m(n\tilde{\beta}_m t)+n\tilde{\beta}_m t)&=&\sum_{m\in\mathcal{I}}2\Vert m\Vert_1\tilde{\beta}_{m} t\\
            &=&\sum_{m\in\mathcal{I}}\lim_{n\to\infty}\Vert m\Vert_1n^{-1}(Y_m(n\tilde{\beta}_m t)+n\tilde{\beta}_m t).
        \end{eqnarray*}
        That is, the limit and the summation can be interchanged. Using this in inequality \eqref{ineqepsilon} and applying \eqref{limsup}, it follows that
        \begin{equation}\label{en0}
            \lim_{n\to\infty}\varepsilon_n(t)=0\quad a.s.
        \end{equation}

        Observe that
        \begin{equation*}
            \Vert X^n(t)-X(t)\Vert_p\leq\Vert X^n(0)-X(0)\Vert_p+\varepsilon_n(t)+\int_0^tM\Vert X^n(s)-X(s)\Vert_p ds.
        \end{equation*}
        Applying Gronwall's inequality to expression above, it yields:
        \begin{equation*}
            \Vert X^n(t)-X(t)\Vert_p\leq(\Vert X^n(0)-X(0)\Vert_p+\varepsilon_n(t))\exp(Mt).
        \end{equation*}
        Since $X^n(0)\to X(0)$ almost surely and equation \eqref{en0} holds, we get that, for every $t\geq0$,
        \begin{equation}
            \lim_{n\to\infty}\sup_{s\leq t}\Vert X^n(s)-X(s)\Vert_p=0\quad a.s.
        \end{equation}
    \end{proof}

\subsection{Proof of Theorem \ref{LLN}}\label{proof2.1}
    Let
    \begin{equation*}
        E_+=\{x=(x_0,x_1,x_2,\dots)\in\ell_1, \Vert x\Vert_1\leq1, x_i\geq0,\forall i\geq0\}.
    \end{equation*} 
    Note that $E_+\subset E:= B_2(0)$ the open ball of radius $2$ centered at the origin in $\ell_1$. We consider the initial condition $x^0\in E_+$.
    
    Applying Theorem \ref{IDLLN} to the sequence of infinite-dimensional Markov chains $n^{-1}\RM^n$ in $E$, where $\RM^n$ is the continuous-time Markov chain with transitions and rates given by equation \eqref{ratesMTRA}, we obtain a limit process 
    
    \begin{equation*}
        R(t)=\lim_{n\to\infty}\frac{1}{n}\RM^n(t)\quad a.s.,
    \end{equation*}
    for any $t\geq0$. 
    
    We will study the proportions of the limit process when it dies, that is, when the limit proportion of spreader individuals is zero. First, observe that $Y^n\geq0$, so $\int_0^tY^n(s)ds$ is non-decreasing in $t$. Let the absorption time of the process $\tau_n:=\inf\{t\geq0:Y^n(t)=0\}$ and define the accelerated time:
    
    \begin{equation*}
        \gamma^n(t):=\inf\left\{u\in[0,\tau_n):\int_0^uY^{n}(s)ds>t\right\},\quad0\leq t\leq \int_0^\infty Y^{n}(s)ds.
    \end{equation*}
    
    It follows that $\int_0^{\gamma^n(t)}Y^n(s)ds=t$. We let the time-changed continuous-time Markov chain $\{\tilde{\RM}(t)\}_{t\geq0}$ given by
    \begin{equation*}
        \tilde{\RM}^n(t):=\RM^n(\gamma^n(t)),\quad \forall t\geq0.
    \end{equation*}
    Thus, $\tilde{\RM}^n$ has transitions and rates:
    \begin{equation}
        \begin{array}{lll}
             \text{Transition}&\text{Rate}&  \\
            m_{1}=-e_{1} & p_0\tilde{X}_1,& \\
            m_{2}=-e_{1}+e_{2} & (1-p_0-p_1)\tilde{X}_1,&\\
            m_{-i}=-e_{i}+e_{0} & q_i\tilde{X}_i, & i=1,\dots\\
            m_{i}=-e_{i-1}+e_{i}& (1-q_i)\tilde{X}_i, & i=3,\dots\\
            m_{0}=-e_{0} & \left(n-1-\sum_{i=1}^{\infty}\tilde{X}_i\right).&
        \end{array}
    \end{equation}
    Let the accelerated absorption time $\tilde{\tau}_n:=\inf\{t\geq0:\tilde{Y}^n(t)\leq0\}$ and notice that
    \begin{equation*}
        \RM^n(\tau_n)=\tilde{\RM}^n(\tilde{\tau}_n).
    \end{equation*}

    Now, consider the functions
    \begin{align*}
            &\beta_{m_1}(x_0,x_1,x_2,\dots)=p_0x_1, \\
            &\beta_{m_2}(x_0,x_1,x_2,\dots)=(1-p_0-p_1)x_1, \\
            &\beta_{m_{-i}}(x_0,x_1,x_2,\dots)=q_{i}x_i,\quad i=1,\dots\\
            &\beta_{m_{i}}(x_0,x_1,x_2,\dots)=(1-q_i)x_i,\quad  i=3,\dots\\
            &\beta_{m_0}(x_0,x_1,x_2,\dots)=\left(1-\sum_{i=1}^{\infty}x_i\right),
    \end{align*}
    and set $F:\ell_1\to\ell_1$ to be the map $x\mapsto\sum_{i\in\Z}\beta_{m_i}(x)$.

    In order to obtain the limits, we define the family of operators ${G=(G_t:t\geq0)}$, $G_t:E\to\ell_1$, with $G_0(x)=x$ and, for $t>0$,
    \begin{equation}
        G_t(x)=x+\int_0^tF(G_s(x))ds,
    \end{equation}
    for $x\in \ell_1$. 
    
    For a fixed initial condition $x^0\in E_+$, let 
    \begin{equation*}
        \tilde{\tau}_\infty=\inf\{t\geq0:G_t(x^0)\notin E_+\}
    \end{equation*}
    and consider 
    \begin{equation*}
        x^\infty=G_{\tilde{\tau}_\infty}(x^0).
    \end{equation*}

    Indeed, the definition of $F$ ensures that $\tau_\infty<\infty$ and, consequently, $x^\infty$ is well defined. This fact is established later in the proof.
    
    Whenever $x^0\in E_+$, $G_t(x^0)$ is the limit process for $n^{-1}\tilde{\RM}^n(t)$ with initial condition $x^0$ by Theorem \ref{IDLLN}. Moreover,
    \begin{equation*}
        \lim_{n\to\infty}\tilde{\tau}_n=\tilde{\tau}_\infty.
    \end{equation*}
    
    Thus, it suffices to show that for $x^0\in E_+$,
    \begin{equation}\label{E0}
        \tilde{\tau}_\infty=\inf\{t>0: G_t(x^0)\in E_0\},
    \end{equation}
    where 
    \begin{equation*}
        E_0=\{x\in E_+:x_0=0\}.
    \end{equation*}
    
    Now observe that for a fixed initial condition $x^0$, the coordinates $x^0_1,x^0_2,\dots,x^0_i$ fully determine the trajectory of the $i$-th coordinate for $i\geq1$. Indeed,
    \begin{equation}
        x_1(t)=x^0_1-\int_0^tx_1(s)ds,
    \end{equation}
    which has solution $x_1(t)=x^0_1\exp(-t)$,
    \begin{equation}
        x_2(t)=x^0_2+\int_0^t(1-p_1-p_0)x_1(s)-x_2(s)ds,
    \end{equation}
    which has solution $x_2(t)=x^0_2\exp(-t)+(1-p_0-p_1)x^0_1t\exp(-t)$, and for $i\geq2$,
    \begin{equation}\label{isystem}
        x_{i+1}(t)=x^0_{i+1}+\int_0^t(1-q_i)x_i(s)-x_{i+1}(s)ds.
    \end{equation}
    To solve the above equation, we rewrite it as an ordinary differential equation
    \begin{equation}\label{iisystem}
        x_{i+1}'=(1-q_i)x_i-x_{i+1},
    \end{equation}
    with initial condition $x_{j}(0)=x^0_j$ for $1\leq j\leq i$.
    
    Then, it is straightforward that whenever $x^0\in E_+$, $x_i(t)\geq0$ for any $t\geq0$ and $i\geq1$. Moreover, if $x(t)\in E_+$ and $x(t+\delta)\notin E_+$ for any $\delta>0$ then $x(t)\in E_0\subset E_+$. 
    
    To finish the proof, we show that $x(t)\in E_0$ implies $x(t+\delta)\notin E_+$ for any $\delta>0$ and $t>0$. Indeed, since $0\leq q_i\leq 1$ for any $i\geq2$, we have that
    \begin{eqnarray}
        \sum_{i=1}^{\infty}\frac{dx_i}{dt}
        &=&-x_1+((1-p_0-p_1)x_1-x_2)\nonumber\\
        &&+\sum_{i\geq3}((1-q_{i-1})x_{i-1}-x_i)\nonumber\\
        &=&-x_1(p_0+p_1)-\sum_{i\geq2}q_ix_i\label{dxdt}\\
        &\leq&0.\nonumber
    \end{eqnarray}
    Then 
    \begin{equation*}
        \frac{d}{dt}\sum_{i=1}^{\infty}x_i=\sum_{i=1}^{\infty}\frac{dx_i}{dt}
    \end{equation*}
    is well defined. Moreover, given $x^0\in E_+$, $\sum_{i=1}^{\infty}x_i(t)\leq1$ for any $t\geq0$. It follows that
    \begin{equation}\label{eqy}
        x_0(t)=C+\int_0^t\Big(-1+\sum_{i=1}^{\infty}(1+q_i)x_i(s)\Big)ds
    \end{equation}
    is finite, therefore it is well defined. Here, $C$ is such that $x_0(0)=y(0)$.
    
    In order to prove that $\tilde{\tau}_\infty$ is finite, we need to show that the sum inside the integral in equation \eqref{eqy} is bounded by a constant $c<1$ for any $s$ sufficiently large. 

    Since $x_i(s)\geq0$ for any $s\geq0$, it holds
    \begin{equation*}
        \int_0^\infty x_i(s)ds\geq\int_0^tx_i(s)ds.
    \end{equation*}
    From equation \eqref{isystem} (or equivalently equation \eqref{iisystem}), we obtain
    \begin{equation}\label{xn}
    \begin{array}{ccc}
        x_i(t)&=&\exp(-t)\bigg[x_1^0\frac{t^{i-1}}{(i-1)!}\Big(1-\sum_{j=0}^{i-1}p_j\Big)\\
        &&+\displaystyle\sum_{k=2}^{i}\frac{x_k^0}{(i-k)!}t^{i-k}\prod_{j=k}^{i-1}(1-q_j)\bigg].
    \end{array}
    \end{equation}
    Since $x^0\in E_+$,
    \begin{eqnarray*}
        \int_0^\infty x_{i}(s)ds&=&\sum_{k=1}^{i}x_k^0\prod_{j=k}^{i}(1-q_j)\\
        &=&\sum_{k=1}^{i}x_k^0\frac{\sum_{j\geq i+1}p_j}{\sum_{j\geq k}p_j}\\
        &\leq&(1-p_0)^{-1}\sum_{j\geq i+1}p_j.
    \end{eqnarray*}

    That is, it attains its maximum for initial condition $x^0=e_1=(0,1,0,\dots)$, so assume the initial condition $e_1$.
    Let $\varepsilon>0$ and let $i_0$ be an integer such that 
    \begin{equation*}
        \sum_{j=1}^{i}p_j\geq1-\varepsilon,
    \end{equation*}
    for any $i\geq i_0$. Observe that
    \begin{eqnarray*}
        \sum_{i=1}^{\infty}(1+q_i)x_i(t)        &\leq&\sum_{i=1}^{i_0}(1+q_i)\exp(-t)\frac{t^{i-1}}{(i-1)!}\Big(1-\sum_{j=0}^{i-1}p_k\Big)\\
        &&+2\varepsilon\exp(-t)\sum_{i=i_0+1}^{\infty}\frac{t^{i-1}}{(i-1)!}
    \end{eqnarray*}
    The first term on right hand side is a product between an exponential and a polynomial so there is $T=T(\varepsilon)$ such that for any $t\geq T$,
    \begin{equation*}
        \exp(-t)\sum_{i=1}^{i_0}(1+q_i)\frac{t^{i-1}}{(i-1)!}\Big(1-\sum_{k=i+1}^{\infty}p_k\Big)<\varepsilon.
    \end{equation*}
    The second term is bounded by
    \begin{equation*}
        2\varepsilon\exp(-t)\sum_{i=0}^{\infty}\frac{t^{i}}{i!}=2\varepsilon
    \end{equation*}

    Now choose $\varepsilon$ so that $3\varepsilon<1$ and the proof is finished.

    \subsubsection{Obtaining the absorption time}

    To obtain $\tilde{\tau}_\infty$, note that
    \begin{eqnarray*}
        \int_0^tx_n(s)ds&=&\int_0^t\exp(-t)\bigg[x_1^0\frac{t^{n-1}}{(n-1)!}\Big(1-\sum_{j=0}^{n-1}p_j\Big)\\
        &&+\sum_{k=2}^{n}\frac{x_k^0}{(n-k)!}t^{n-k}\prod_{j=k}^{n-1}(1-q_j)\bigg]ds\\
        &=&x_1^0\frac{\gamma(n,t)}{\Gamma(n)}\Big(1-\sum_{j=0}^{n-1}p_j\Big)+\sum_{k=2}^{n}x_k^0\frac{\gamma(n-k+1,t)}{\Gamma(n-k+1)}\prod_{j=k}^{n-1}(1-q_j).
    \end{eqnarray*}

    Applying it to equation \eqref{eqy}, and recalling that $y(\tilde{\tau}_\infty)=0$, equation \eqref{eqtau} is obtained.

\subsection{Proof of Theorem \ref{Maxprop}}\label{proof2.2}
Observe that the time-changed continuous-time Markov Chain $\{\tilde{\RM}^n(t)\}_{t\geq0}$ has the same proportions of the original continuous-time Markov Chain $\{\RM^n(t)\}_{t\geq0}$ for a different time scale, then the maximal proportions in the sense of Theorem \ref{Maxprop} of the original process may be characterized through the accelerated process, provided that there is uniqueness (and existence) of the limit time $\tilde{\tau}_{\infty}$.

Recall that
\begin{equation*}
    x_{0}'(t)=-1+\sum_{i=1}^{\infty}(1+q_i)x_i(t)
\end{equation*}

Also, we have shown that there exists $T>0$ depending on $\varepsilon>0$ and on the distribution such that 
\begin{equation*}
    \sum_{i=1}^{\infty}(1+q_i)x_i(t)\leq c<1
\end{equation*}
for any $t\geq T$.  

By hypothesis $x_{0}'(0)>0$, and the series is absolutely convergent, then there is $\zeta\in(0,\tilde{\tau}_\infty)$ such that $x_{0}'(\zeta)=0$ and $x_{0}(\zeta)=\sup_{t\in(0,\tau_{\max})}x_{0}(t)$. For the second part, notice that since $x_i(t)=\exp(-t)p_i(t)$ with $p_i$ being a polynomial of degree $i-1$, we study the signs of
\begin{equation}\label{eqimport}
 -\exp(t)+\sum_{i=1}^{\infty}(1+q_i)p_i(t).
\end{equation}

We may write the sum as
\begin{eqnarray*}
    p(t):=\sum_{i=1}^{\infty}(1+q_i)p_i(t)&=&\sum_{i=1}^{\infty}x_i^0(1+q_i)\\
    &&+\sum_{\ell=1}^{\infty}\frac{t^{\ell}}{\ell!}\bigg[\Big(1-\sum_{j=0}^{\ell}p_j\Big)(1+q_{\ell+1})x^0_1\\
    &&+\sum_{k=2}^{\infty}(1+q_{\ell+k})x^0_k\prod_{m=1}^{\ell}(1-q_{k+m})\bigg]\\
    &&=\sum_{i=0}^{\infty}a_i\frac{t^i}{i!}.
\end{eqnarray*}

Notice that for $\ell\geq1$,
\begin{equation*}
    (1+q_{\ell+k})\prod_{m=1}^{\ell}(1-q_{k+m})\leq1,
\end{equation*}
with equality if and only if $p_{\ell+r}=0$ for any $r\geq1$, and
\begin{equation}\label{pjcondition}
    \Big(1-\sum_{j=0}^{\ell}p_j\Big)(1+q_{\ell+1})=1+p_{\ell+1}-\sum_{j=0}^{\ell}p_j,
\end{equation}
so there is an integer $k$ such that $a_i<1$ for any $i\geq k$. Then
\begin{equation}\label{derivativexp}
    \exp(t)\geq\frac{d^ip}{dt^i}(t)
\end{equation}
for any $t\geq0$ and $i\geq k$. Since $\exp(t)$ is equal to all its derivatives, it follows that $p(t)-\exp(t)$ changes sign at most $k-1$ times between $0$ and $\tau_{\infty}$. Moreover, there is a global maximum in the interval $[0,\tau_{\infty{}}]$, say $\tilde{\tau}_{\max{}}$. The proof that $\tilde{\tau}_{\max{}}^{(n)}\to\tilde{\tau}_{\max{}}:=\zeta$ follows from Theorem \ref{IDLLN}. 

For the last part of the proof, note that the conditions $p_{m+1}\geq\sum_{j=0}^{m}p_j$ for any $m\leq m_0$ and $p_{m}=0$ for any $m\geq m_0+1$ guarantee that the $m$-th derivative of $y$ satisfies $y^{(m)}\geq y^{(m+1)}$ for any $m$, so the uniqueness of the maximum of $y$ is established.

Also, note that $p_{m+1}\leq\sum_{j=0}^{m}p_j$ for any $m\geq0$ with equation \eqref{pjcondition} also establishes the uniqueness of the maximum and the proof is complete.

\section*{Acknowledgements}
This study was financed in part by the Coordenação de Aperfeiçoamento de Pessoal de Nível Superior - Brasil (CAPES) - Finance Code 001. It was also supported by grants 2022/08948-2, 2023/13453-5 and 2025/02013-0 S\~ao Paulo Research Foundation (FAPESP). D.A.L. thanks Daniel Miranda (UFABC) for illuminating discussions.

\bibliographystyle{siam}
\bibliography{Biblio.bib}

\begin{description}
  \item[Cristian F. Coletti and Alejandra Rada] Centro de Matemática, Computação e Cognição - Universidade Federal do ABC  \\
        Santo André, SP, Brazil \\
        E-mails: \texttt{cristian.coletti@ufabc.edu.br}, \texttt{alejandra.rada@ufabc.edu.br}
        
  \item[Denis A. Luiz] Instituto de Matemática, Estatística e Computação Científica - Universidade Estadual de Campinas \\
        Campinas, SP, Brazil \\
        \texttt{denisalu@ime.unicamp.br}
\end{description}

\end{document}